\documentclass{amsart}

\usepackage{amsmath,amssymb,amsthm}
\usepackage{enumerate}
\usepackage{graphicx,caption}
\usepackage[caption=false]{subfig}

\theoremstyle{plain}
\newtheorem{theorem}{Theorem}[section]
\newtheorem{lemma}[theorem]{Lemma}
\newtheorem{proposition}[theorem]{Proposition}
\newtheorem{conjecture}[theorem]{Conjecture}

\theoremstyle{definition}

\newtheorem{problem}[theorem]{Problem}

\newcommand{\rhat}{\hat{r}}
\newcommand{\rc}{{\hat{r}}_c}

\begin{document}

\title[Size Ramsey numbers for matchings]{Two types of size Ramsey numbers for matchings of small order}
\author[V. Vito \and D. R. Silaban]{Valentino Vito \and Denny Riama Silaban}

\address{Department of Mathematics\\
Faculty of Mathematics and Natural Sciences (FMIPA)\\
Universitas Indonesia\\
Depok 16424\\
Indonesia}
\email{valentino.vito@sci.ui.ac.id\\
denny@sci.ui.ac.id}


\begin{abstract}
For simple graphs $G$ and $H$, their size Ramsey number $\hat{r}(G,H)$ is the smallest possible size of $F$ such that for any red-blue coloring of its edges, $F$ contains either a red $G$ or a blue $H$. Similarly, we can define the connected size Ramsey number ${\hat{r}}_c(G,H)$ by adding the prerequisite that $F$ must be connected. In this paper, we explore the relationships between these size Ramsey numbers and give some results on their values for certain classes of graphs. We are mainly interested in the cases where $G$ is either a $2K_2$ or a $3K_2$, and where $H$ is either a cycle $C_n$ or a union of paths $nP_m$. Additionally, we improve an upper bound regarding the values of $\hat{r}(tK_2,P_m)$ and ${\hat{r}}_c(tK_2,P_m)$ for certain $t$ and $m$.
\end{abstract}

\maketitle

\section{Introduction}

Let $F$, $G$ and $H$ be simple graphs. We write $v(F)$ and $e(F)$ to denote the order and size of $F$, respectively. We write $F\to (G,H)$ if for every red-blue coloring of edges in $F$, there exists either a red $G$ or a blue $H$ in $F$. If $F\to (G,H)$, we say that $F$ is an \emph{arrowing graph} of $G$ and $H$. A \emph{$(G,H)$-coloring} of $F$ is a red-blue coloring of edges in $F$ such that $F$ contains neither a red $G$ nor a blue $H$. Thus $F\not\to (G,H)$ means that $F$ admits a $(G,H)$-coloring.

The \emph{size Ramsey number} of $G$ and $H$, denoted by $\rhat(G,H)$, is defined as the smallest possible size of a graph $F$ such that $F\to (G,H)$ holds \cite{Erd78}. A variant of the size Ramsey number, the \emph{connected size Ramsey number} $\rc(G,H)$ of $G$ and $H$, is the smallest possible size of a connected graph $F$ such that $F\to (G,H)$ holds \cite{Rah15}. Both of these numbers are considered in this paper, and we try to explore some relationships between them.

We focus on \emph{matchings} $G=tK_2$ whose order $2t$ is small. More precisely, we limit ourselves to the cases where $G$ is either a $2K_2$ or a $3K_2$, and where $H$ is either a cycle $C_n$ or a union of paths $nP_m$. We shall see that this allows us to obtain nontrivial results involving large order matchings, as Theorem \ref{r(tK_2,P_m) sharp} illustrates. For some previous research discussing other cases of $G$ and $H$, see \cite{Ben12,Lor98,Rah16,Rah17}.

In previous studies, Erd\H{o}s and Faudree \cite{Erd81} discussed the size Ramsey numbers of graphs involving matchings $tK_2$, $t \ge 1$. They showed that \[\rhat(tK_2,P_4)=\left\lceil \frac{5t}{2}\right\rceil \quad \text{and} \quad \rhat(tK_2,P_5)=\begin{cases}
3t, & t \text{ even,} \\ 3t+1, & t \text{ odd.}
\end{cases}\]
Furthermore, they proved that there is a constant $c$ depending on $t$ such that $\rhat(tK_2,C_n)\le n+c\sqrt{n}$. Silaban \textit{et al.} \cite{SilPP}, along with Vito \textit{et al.} \cite{Vit21} provided a proof that $\rhat(2K_2,P_m)=\rc(2K_2,P_m)=m+1$ and that $C_{m+1}\to (2K_2,P_m)$.

In \cite{Ass19}, Assiyatun \textit{et al.} managed to prove that $\rc(2K_2,2P_m)=2m+1$. We are able to generalize this result to $\rc(2K_2,nP_m)=nm+1$. This result motivates us to work on the value of $\rhat(2K_2,nP_m)$ by applying our knowledge of the value of $\rc(2K_2,nP_m)$.

On the other hand, Rahadjeng \textit{et al.} \cite{Rah15} computed the exact values of $\rc(tK_2,P_4)$ for small $t$, and they showed that
\[\rc(tK_2,P_4)\le\begin{cases}
3t-1, & t \text{ even,} \\ 3t, & t \text{ odd.}
\end{cases}\]
They further claimed that $\rc(2K_2,C_n)=2n$ for $n \ge 4$. However, if $n$ is sufficiently large, we are able to construct a connected arrowing graph $F$ of $2K_2$ and $C_n$ such that $e(F)<2n$, refuting this claim.

For this paper, we determine the exact value of $\rc(2K_2,nP_m)$ and an upper bound for $\rhat(2K_2,nP_m)$ for $n\ge 1$ and $m\ge 3$. We also give exact values of $\rhat(2K_2,nP_m)$ for small values of $n$. Furthermore, we consider the connected size Ramsey numbers for pairs of graphs $(3K_2,P_m)$ and $(2K_2,C_n)$. We also manage to sharpen a previous upper bound \cite{Vit21} for both $\rhat(tK_2,P_m)$ and $\rc(tK_2,P_m)$ when $t \ge 3$ is odd and $m \ge 9$.

\section{Preliminaries}

It is easy to see that $\rhat(G,H) \le \rc(G,H)$. We can also see that if $F_1\to (G_1,H)$ and $F_2\to (G_2,H)$, then their disjoint union satisfies $F_1+F_2\to (G_1+G_2,H)$. By induction, we have that $F\to (G,H)$ implies $tF\to (tG,H)$, $t\ge 1$. In addition, we have the following lemma.

\begin{lemma}[\cite{Sil20}]\label{2K_2}
Let $H$ be a graph. Then $F\to (2K_2,H)$ holds if and only if the following conditions are satisfied:
\begin{enumerate}
  \item $H \subseteq F-v$ for every $v \in V(F)$ and
  \item $H \subseteq F-C_3$ for every $C_3$ in $F$.
\end{enumerate}
\end{lemma}

In our main results, Lemma \ref{2K_2} is applied to prove that, given a graph $H$, $F\to (2K_2,H)$ is satisfied for some graph $F$. We also consider the following theorem which provides some upper bounds for the values of $\rhat(tK_2,P_m)$ and $\rc(tK_2,P_m)$.

\begin{theorem}[\cite{Vit21}]\label{r(tK_2,P_m)}
For $t\ge 1$, $m\ge 3$,\[\rhat(tK_2,P_m)\le\begin{cases}
\frac{t(m+1)}{2}, & t \text{ even,} \\ \frac{(t+1)(m+1)}{2}-2, & t \text{ odd.}\end{cases}\] 
and
\[\rc(tK_2,P_m)\le\begin{cases}
\frac{t(m+2)}{2}-1, & t \text{ even,} \\ \frac{(t+1)(m+2)}{2}-3, & t \text{ odd.}\end{cases}\] 
\end{theorem}

The upper bounds of Theorem \ref{r(tK_2,P_m)} are sharpened for certain values of $t$ and $m$ in Theorem \ref{r(tK_2,P_m) sharp}.

\section{$2K_2$ versus $nP_m$}

It has previously been established that $\rhat(2K_2,P_m)=\rc(2K_2,P_m)=m+1$ \cite{SilPP}. Now, we consider pairs of graphs in the more general form of $(2K_2,nP_m)$ and present the following theorem as our first result.

\begin{theorem}\label{r(2K_2,nP_m)}
For $n\ge 1$, $m\ge 3$,
\[\rhat(2K_2,nP_m)\le \mathrm{min}\{nm+1,(n+1)(m-1)\}.\]
\end{theorem}

\begin{figure}
\centering
\captionsetup{width=.3\textwidth}
\subfloat[The graph $F=C_{nm+1}$ satisfies $F\to (2K_2,nP_m)$.]{%
\includegraphics[width=.3\textwidth]{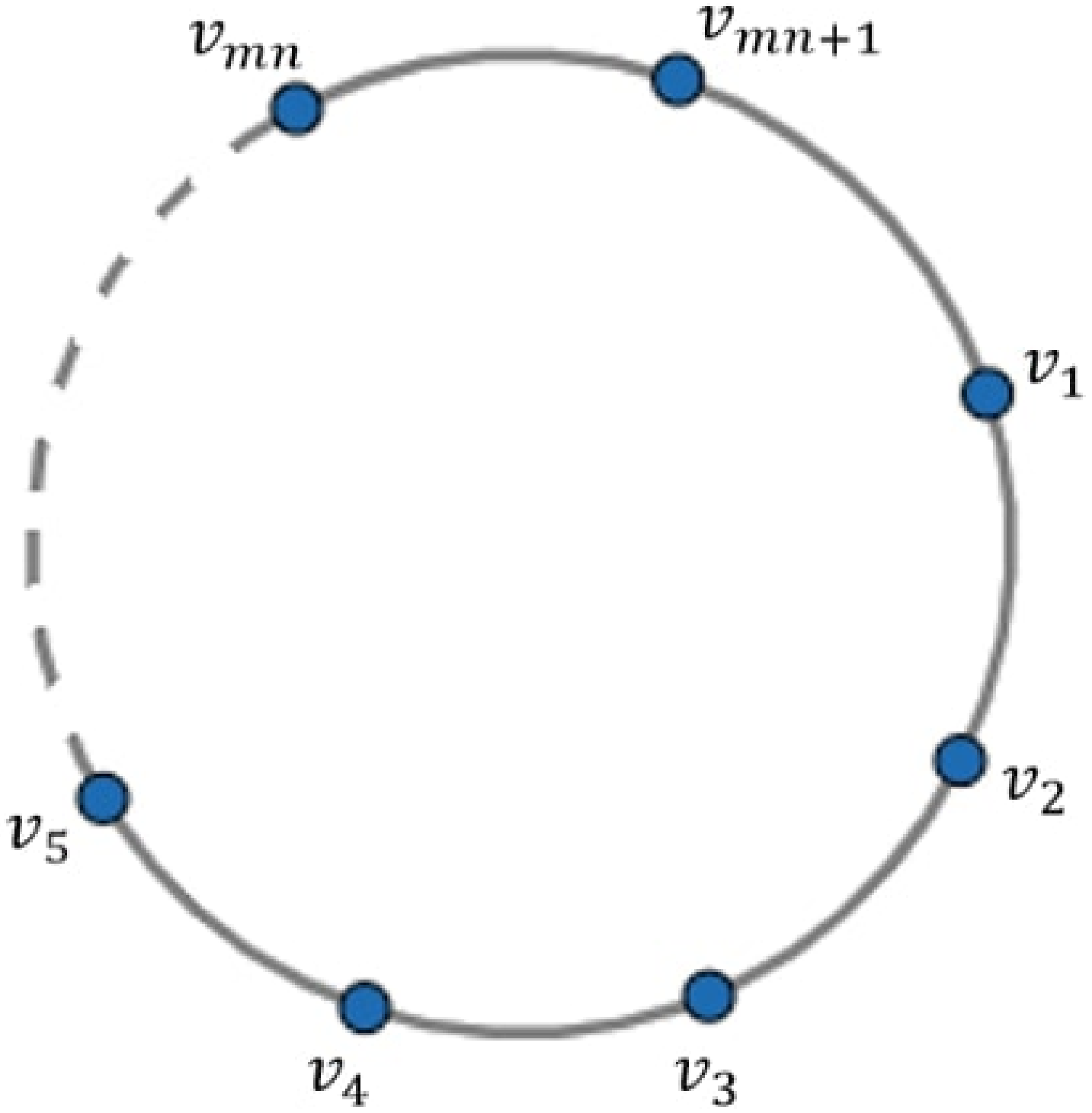}}\hspace{16pt}
\subfloat[The graph $G=(n+1)P_m$ satisfies $G\to (2K_2,nP_m)$.]{%
\includegraphics[width=.3\textwidth]{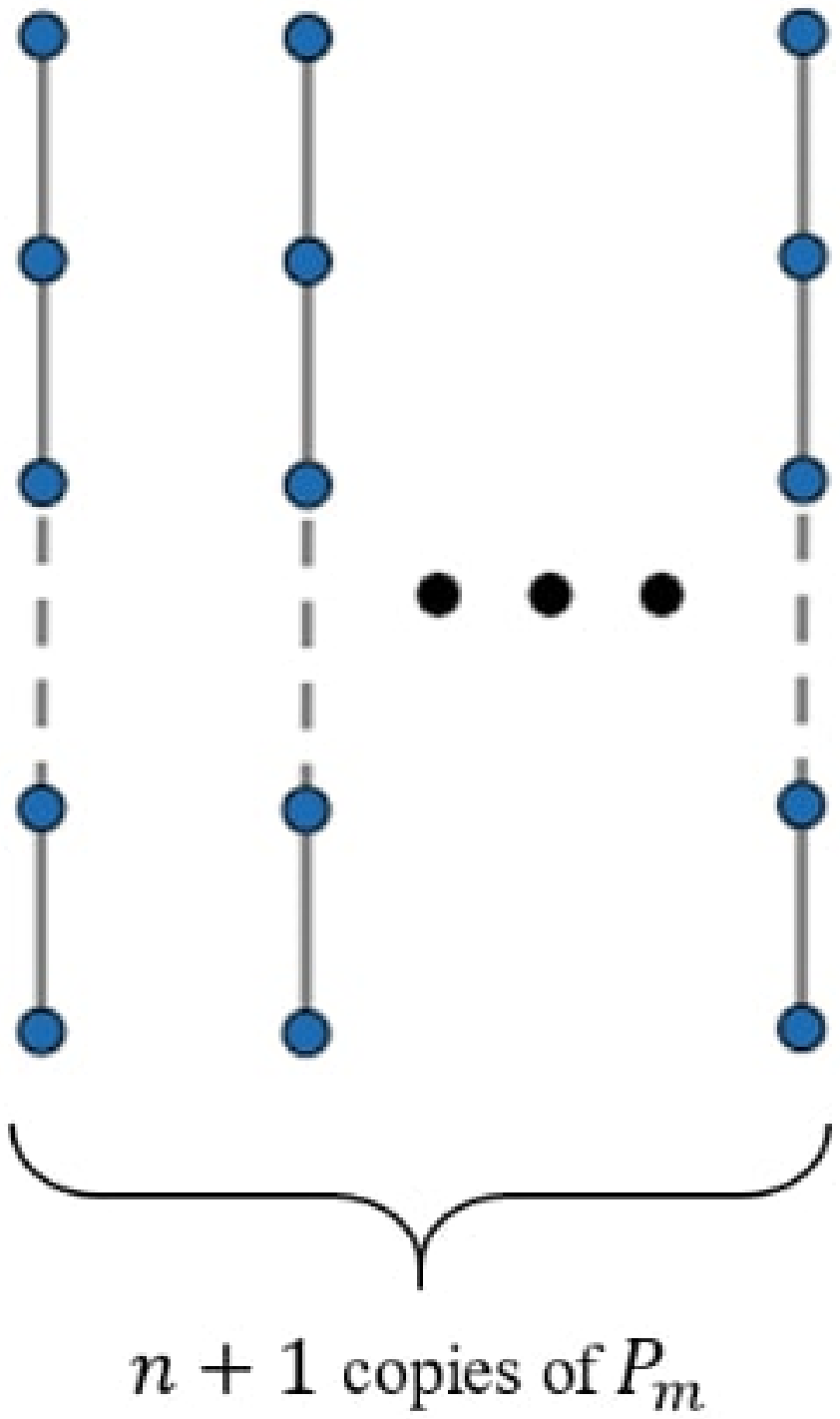}}
\captionsetup{width=.5\textwidth}
\caption{Arrowing graphs of $2K_2$ and $nP_m$.} \label{Graph12}
\end{figure}

\begin{proof}
We first show that $\rhat(2K_2,nP_m)\le nm+1$. Consider the graph $F=C_{nm+1}$ of size $nm+1$ as in Figure \ref{Graph12}(A). Notice that for every $v \in V(F)$, the graph $F-v$ is the path $P_{nm}$. Thus, we have $nP_m \subseteq P_{nm}=F-v$. Since $F$ does not contain a $C_3$, we conclude that $F\to (2K_2,nP_m)$ by Lemma \ref{2K_2}. It follows that $\rhat(2K_2,nP_m)\le nm+1$.

Now we show that $\rhat(2K_2,nP_m)\le (n+1)(m-1)$. Consider the graph $G=(n+1)P_m$ of size $(n+1)(m-1)$ as in Figure \ref{Graph12}(B). We see that for every $v \in V(G)$, the graph $G-v$ contains $nP_m$. Moreover, since $G$ does not contain a $C_3$, we conclude that $G\to (2K_2,nP_m)$ by Lemma \ref{2K_2}. It follows that $\rhat(2K_2,nP_m)\le (n+1)(m-1)$.
\end{proof}

While we only managed to provide an upper bound for $\rhat(2K_2,nP_m)$ in Theorem \ref{r(2K_2,nP_m)}, we are able to acquire the exact value of $\rc(2K_2,nP_m)$ in the following theorem.

\begin{theorem}\label{r_c(2K_2,nP_m)}
For $n\ge 1$, $m\ge 3$,
\[\rc(2K_2,nP_m)=nm+1.\]
\end{theorem}

\begin{proof}
By the proof of Theorem \ref{r(2K_2,nP_m)}, we see that $C_{nm+1}\to (2K_2,nP_m)$. Therefore, $\rc(2K_2,nP_m)\le nm+1$. We now show that $\rc(2K_2,nP_m)\ge nm+1$. Suppose that $F$ is a connected graph of size $nm$. We pick a vertex $u \in V(F)$ depending on whether $F$ contains a cycle. If $F$ is a tree, then we choose a vertex $u$ which is adjacent to a leaf vertex. Otherwise, we choose a vertex $u$ which is part of a cycle. Denote $d$ as the degree of $u$ in either case. 

Color every edge incident to $u$ red and $F-u$ blue. Clearly, $F$ contains no red $2K_2$. The graph $F-u$ contains at most $d-1$ components and $nm-d$ edges. It follows that $F-u$ contains at most $(nm-d)+(d-1)=nm-1$ vertices, and thus it is not possible for $F-u$ to contain an $nP_m$. Therefore, $F$ admits a $(2K_2,nP_m)$-coloring. Since $F$ is an arbitrary graph of size $nm$, we have $\rc(2K_2,nP_m)\ge nm+1$.
\end{proof}

We have found that finding the value of $\rc(2K_2,nP_m)$ is easier than finding the value of $\rhat(2K_2,nP_m)$. However, we can determine the exact values of $\rhat(2K_2,nP_m)$ for $n=2,3,4$ which turn out to be precisely the upper bound obtained in Theorem \ref{r(2K_2,nP_m)}. We note that Theorem \ref{r_c(2K_2,nP_m)} is applied to find these size Ramsey values.

\begin{proposition}\label{r(2K_2,2P_m)}
For $m\ge 3$,\[\rhat(2K_2,2P_m)=\mathrm{min}\{2m+1,3m-3\}=\begin{cases} 6, & m=3, \\ 2m+1, & m\ge 4.
\end{cases}\]
\end{proposition}

\begin{proof}
Let $k=\mathrm{min}\{2m+1,3m-3\}$. We need to show that $\rhat(2K_2,2P_m)\ge k$, so suppose that $F$ is a graph of size less than $k$ (that is, $e(F)<k$). We show that $F\not\to (2K_2,2P_m)$ by defining a $(2K_2,2P_m)$-coloring on $F$. We can assume that every component of $F$ contains a $P_m$ since coloring the components of $F$ without a $P_m$ blue would still produce a $(2K_2,2P_m)$-coloring. So it follows that $F$ has at most two components since otherwise, $e(F)\ge 3(m-1)\ge k$. The case when $F$ has only one component is dealt via Theorem \ref{r_c(2K_2,nP_m)}, so we can assume that $F$ has exactly two components.

Let $F_1$ and $F_2$ be the components of $F$. We cannot have both $F_1\to (2K_2,P_m)$ and $F_2\to (2K_2,P_m)$ since this would imply that $e(F)\ge 2(m+1)\ge k$ by Theorem \ref{r_c(2K_2,nP_m)}. Thus, we assume that $F_1\not\to (2K_2,P_m)$. Color $F_1$ by a $(2K_2,P_m)$-coloring and $F_2$ blue. To prove that this is a $(2K_2,2P_m)$-coloring of $F$, we need to show that $F_2$ does not contain a $2P_m$. But if $F_2$ contains a $2P_m$, then $e(F_2)\ge 2(m-1)$, so $e(F)\ge 3(m-1)\ge k$. Hence, we have a $(2K_2,2P_m)$-coloring of $F$. Since $F$ is an arbitrary graph of size less than $k=\mathrm{min}\{2m+1,3m-3\}$, we have $\rhat(2K_2,2P_m)\ge \mathrm{min}\{2m+1,3m-3\}$, and thus the theorem holds.
\end{proof}

\begin{proposition}\label{r(2K_2,3P_m)}
For $m\ge 3$,\[\rhat(2K_2,3P_m)=\mathrm{min}\{3m+1,4m-4\}=\begin{cases} 4m-4, & m=3,4, \\ 3m+1, & m\ge 5.
\end{cases}\]
\end{proposition}

\begin{proof}
Let $k=\mathrm{min}\{3m+1,4m-4\}$. We need to show that $\rhat(2K_2,3P_m)\ge k$, so suppose that $F$ is a graph of size less than $k$ whose components each contain a $P_m$. We show that $F\not\to (2K_2,3P_m)$. Following the first paragraph of the proof of Proposition \ref{r(2K_2,2P_m)}, we only need to consider the cases where $F$ has two or three components. The case where $F$ has three components is treated similarly to the second paragraph of the proof of Proposition \ref{r(2K_2,2P_m)}. Now let $F_1$ and $F_2$ be the components of $F$. Consider the following two cases.

\textbf{Case 1.} One of the components, say $F_1$, satisfies $F_1\to (2K_2,2P_m)$. It follows that $F_2\not\to (2K_2,P_m)$ since otherwise, $e(F)\ge (2m+1)+(m+1)\ge k$ by Theorem \ref{r_c(2K_2,nP_m)}. Color $F_1$ blue and $F_2$ by a $(2K_2,P_m)$-coloring. Observe that $F_1$ cannot contain a $3P_m$ since otherwise, $e(F_1)\ge 3(m-1)$, implying that $e(F)\ge 4(m-1)\ge k$. Hence, this coloring is indeed a $(2K_2,3P_m)$-coloring of $F$.

\textbf{Case 2.} $F_1\not\to (2K_2,2P_m)$ and $F_2\not\to (2K_2,2P_m)$. It is not possible for both $F_1$ and $F_2$ to contain a $2P_m$ since this would imply that $e(F)\ge 4(m-1)\ge k$. We can then choose a component, say $F_1$, which does not contain a $2P_m$. Color $F_1$ blue and $F_2$ by a $(2K_2,2P_m)$-coloring. This produces a $(2K_2,3P_m)$-coloring of $F$.

Either way, $F$ admits a $(2K_2,3P_m)$-coloring. Since $F$ is an arbitrary graph of size less than $k=\mathrm{min}\{3m+1,4m-4\}$, we have $\rhat(2K_2,3P_m)\ge \mathrm{min}\{3m+1,4m-4\}$, and thus the theorem holds.
\end{proof}

\begin{proposition}\label{r(2K_2,4P_m)}
For $m\ge 3$,\[\rhat(2K_2,4P_m)=\mathrm{min}\{4m+1,5m-5\}=\begin{cases} 5m-5, & m=3,4,5, \\ 4m+1, & m\ge 6.
\end{cases}\]
\end{proposition}

\begin{proof}
Let $k=\mathrm{min}\{4m+1,5m-5\}$. We need to show that $\rhat(2K_2,4P_m)\ge k$, so suppose that $F$ is a graph of size less than $k$ whose components each contain a $P_m$. We show that $F\not\to (2K_2,4P_m)$. Following the proof of Proposition \ref{r(2K_2,2P_m)}, we only need to consider the cases where $F$ has two to four components. The cases where $F$ has three or four components are treated similarly to Proposition \ref{r(2K_2,3P_m)} and Proposition \ref{r(2K_2,2P_m)}, respectively.

It remains to consider the case where $F$ has two components, say $F_1$ and $F_2$. We cannot have both $F_1\to (2K_2,2P_m)$ and $F_2\to (2K_2,2P_m)$ since this would imply that $e(F)\ge 2(2m+1)\ge k$ by Theorem \ref{r_c(2K_2,nP_m)}, so we assume that $F_1\not\to (2K_2,2P_m)$. However, we must still have $F_1\to (2K_2,P_m)$. It follows that $F_2\not\to (2K_2,3P_m)$ since otherwise, $e(F)\ge (3m+1)+(m+1)\ge k$ by Theorem \ref{r_c(2K_2,nP_m)}. Consider the following two cases.

\textbf{Case 1.} $F_2$ contains a $3P_m$. This implies that $F_1$ does not contain a $2P_m$. Color $F_1$ blue and $F_2$ by a $(2K_2,3P_m)$-coloring. This produces a $(2K_2,4P_m)$-coloring of $F.$

\textbf{Case 2.} $F_2$ does not contain a $3P_m$. Color $F_1$ by a $(2K_2,2P_m)$-coloring and $F_2$ blue. This produces a $(2K_2,4P_m)$-coloring of $F$.

In both cases, $F$ admits a $(2K_2,4P_m)$-coloring. Since $F$ is an arbitrary graph of size less than $k=\mathrm{min}\{4m+1,5m-5\}$, we have $\rhat(2K_2,4P_m)\ge \mathrm{min}\{4m+1,5m-5\}$, and thus the theorem holds.
\end{proof}

We are not able to obtain the exact value of $\rhat(2K_2,nP_m)$ in general. But based on the preceding results on small values of $n$, the following conjecture can be posed. 

\begin{conjecture}
For $n\ge 5$, $m\ge 3$,
\[\rhat(2K_2,nP_m)=\mathrm{min}\{nm+1,(n+1)(m-1)\}.\]
\end{conjecture}

\section{$3K_2$ versus $P_m$}

We now turn to pairs of graphs in the form $(3K_2,P_m)$. First, we present the following lemma, which is similar in nature to Lemma \ref{2K_2}.

\begin{lemma}\label{3K_2}
Let $H$ be a graph and suppose $F$ does not contain a cycle of order $5$ or less. Then $F\to (3K_2,H)$ holds if and only if $H \subseteq F-\{u,v\}$ for every $u,v \in V(F)$.
\end{lemma}

\begin{proof}
Suppose that $F-\{u,v\}$ does not contain $H$ for some $u,v \in V(F)$. Then by coloring the edges incident to at least one of $v$ and $w$ red, and coloring $F-\{u,v\}$ blue will produce a $(3K_2,H)$-coloring of $F$. This implies that $F\not\to (3K_2,H)$.

Conversely, suppose that $F$ admits a $(3K_2,H)$-coloring. Let $F'$ be the red subgraph of $F$ with respect to this coloring. We see that $F'$ contains neither a $3K_2$ nor a cycle of order $5$ or less. We claim that $F'$ is a union of at most two stars. If $F'$ is disconnected, then it must be a disjoint union of two stars, so assume that it is connected. Let $P_k$ be the longest path contained in $F'$. Since $F'$ does not contain a $3K_2$, we must have $k\le 5$ since $P_6$ contains a $3K_2$. If $k \le 3$, then $F'$ is easily shown to be a star. Similarly, it is also easy to see that $F'$ is a union of two stars when $4 \le k \le 5$.

We have just shown that $F'$ is a union of at most two stars. Therefore, there are vertices $u,v \in V(F')$ such that $F'-\{u,v\}$ is empty. In other words, $F-\{u,v\}$ is contained in the blue subgraph $F''$ of $F$. Since $H \nsubseteq F''$, we have that $H \nsubseteq F-\{u,v\}$, and the proof is complete.
\end{proof}

With the above lemma, we are able to give an arrowing graph of $3K_2$ and $P_m$. This gives us an upper bound for $\rc(3K_2,P_m)$.

\begin{proposition}\label{r_c(3K_2,P_m)}
For $m\ge 9$,
\[\rc(3K_2,P_m)\le \left\lceil\frac{3m+7}{2}\right\rceil.\]
\end{proposition}

\begin{figure}
\centering
\captionsetup{width=.35\textwidth}
\subfloat[The graph $F$ satisfies $F\to (3K_2,P_m)$ when $m \ge 9$ is odd.]{%
\includegraphics[width=.36\textwidth]{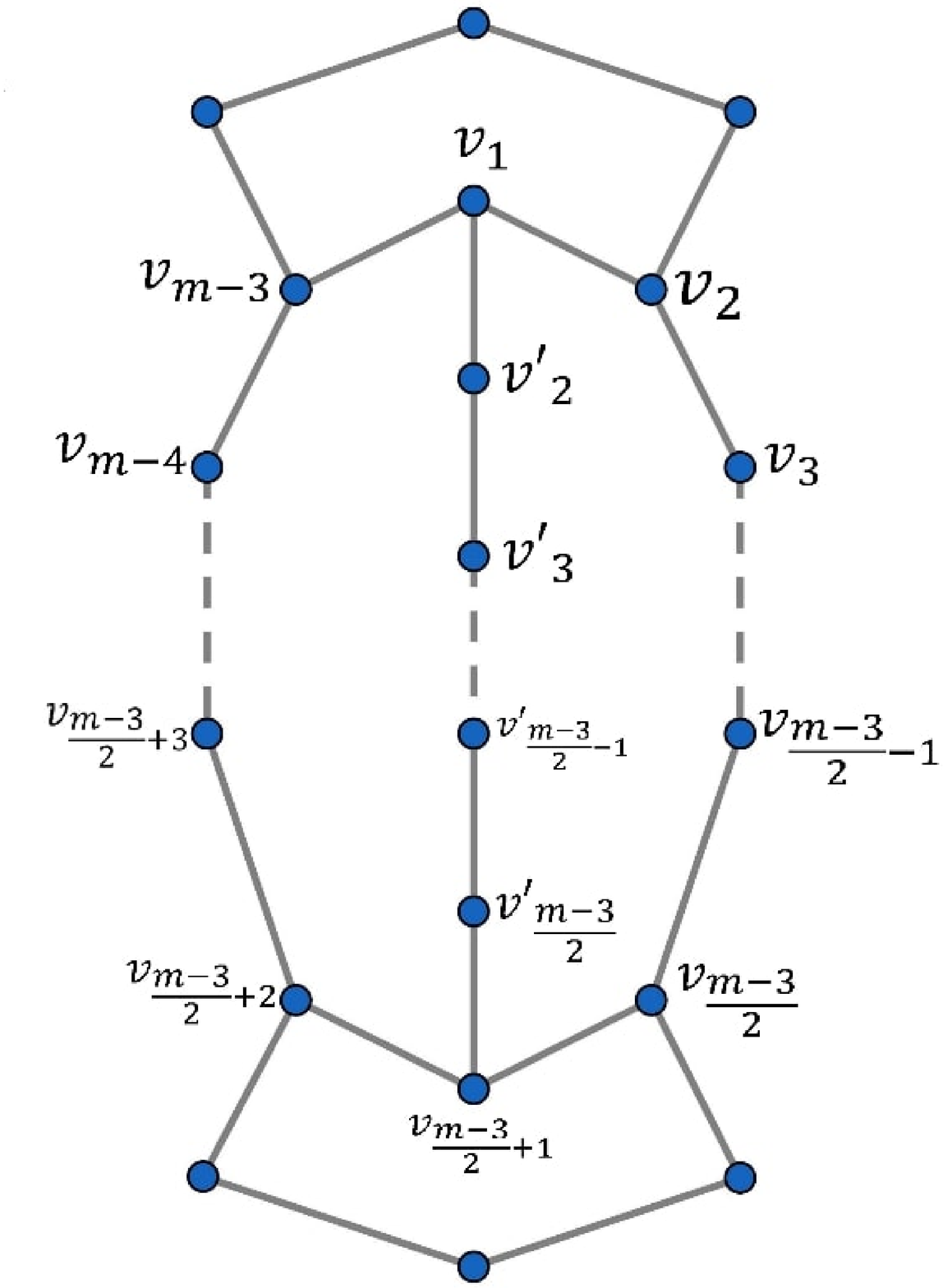}}\hspace{16pt}
\subfloat[The graph $G$ satisfies $G\to (3K_2,P_m)$ when $m \ge 10$ is even.]{%
\includegraphics[width=.36\textwidth]{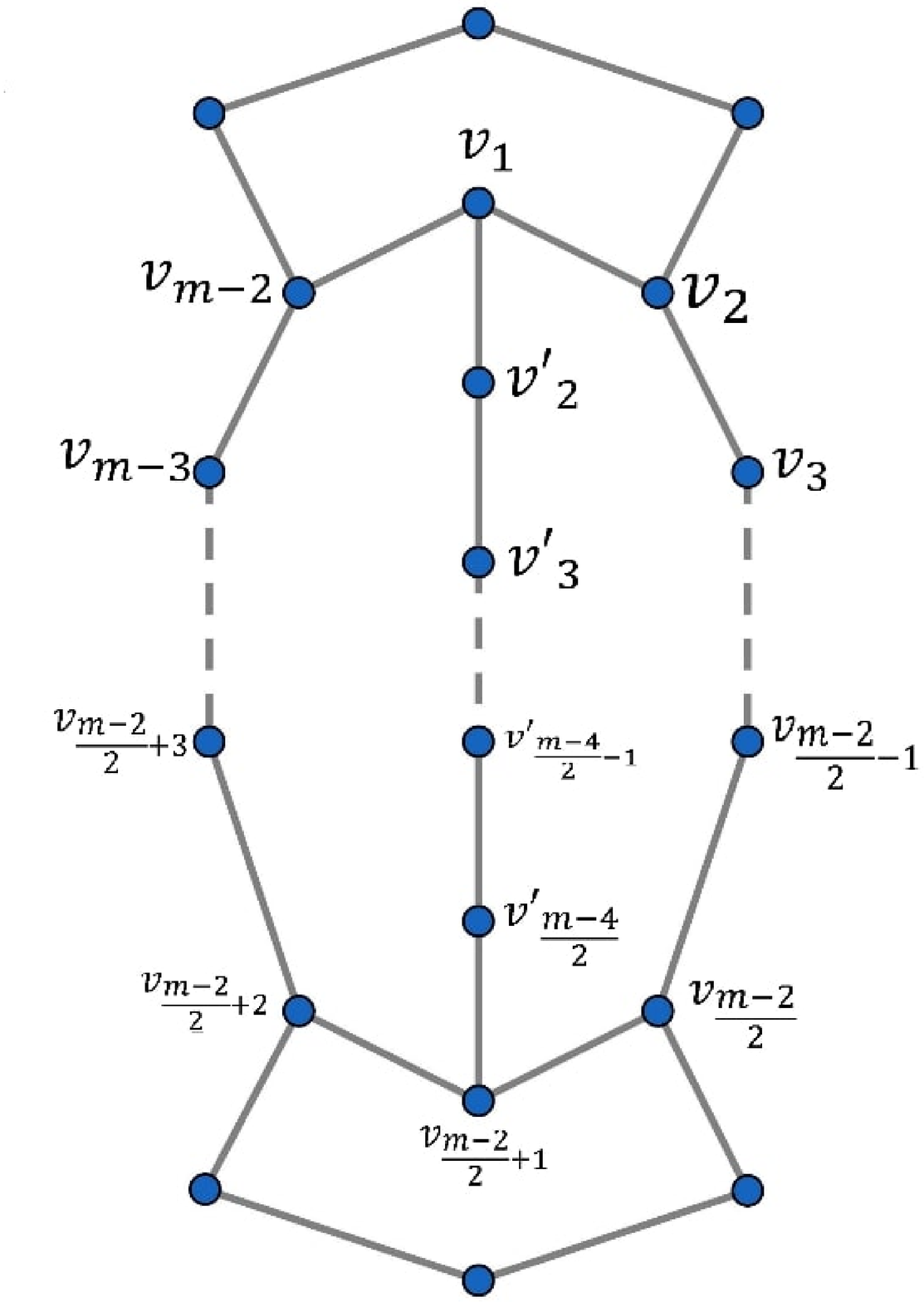}}
\captionsetup{width=.5\textwidth}
\caption{Arrowing graphs of $3K_2$ and $P_m$.} \label{Graph34}
\end{figure}

\begin{proof}
Suppose $m$ is odd. Let $F$ be the connected graph of size $\left\lceil\frac{3m+7}{2}\right\rceil$ shown in Figure \ref{Graph34}(A). Let $u,v \in V(F)$. It is not hard to verify that $F-u$ contains a cycle $C$ of order at least $m+1$ regardless of the vertex $u$ taken. Since $C\to (2K_2,P_m)$, we have $F-u\to (2K_2,P_m)$. By Lemma \ref{2K_2}, we see that $P_m \subseteq (F-u)-v = F-\{u,v\}$. Since $F$ contains no cycle of order $5$ or less, we can conclude from Lemma \ref{3K_2} that $F\to (3K_2,P_m)$, and thus $\rc(3K_2,P_m)\le \left\lceil\frac{3m+7}{2}\right\rceil$ when $m$ is odd.

Now suppose $m$ is even. Let $G$ be the connected graph of size $\left\lceil\frac{3m+7}{2}\right\rceil$ shown in Figure \ref{Graph34}(B). Since $G-u$ contains a cycle of order at least $m+1$ for every $u \in V(G)$, a repeat of the previous argument shows that $\rc(3K_2,P_m)\le \left\lceil\frac{3m+7}{2}\right\rceil$ when $m$ is even. The proof is then complete.
\end{proof}

Proposition \ref{r_c(3K_2,P_m)} can be applied to sharpen the bounds in Theorem \ref{r(tK_2,P_m)} when $t\ge 3$ is odd and $m$ is sufficiently large. This is done by constructing an arrowing graph of $tK_2$ and $P_m$ containing a connected arrowing graph $F\to (3K_2,P_m)$ of size $\left\lceil \frac{3m+7}{2}\right\rceil$ guaranteed to exist by Proposition \ref{r_c(3K_2,P_m)}.

\begin{theorem}\label{r(tK_2,P_m) sharp}
For odd $t\ge 3$ and $m\ge 9$,
\[\rhat(tK_2,P_m)\le \left\lceil\frac{3m+7}{2}\right\rceil+\frac{(t-3)(m+1)}{2}\]
and
\[\rc(tK_2,P_m)\le \left\lceil\frac{3m+7}{2}\right\rceil+\frac{(t-3)(m+2)}{2}.\]
\end{theorem}

\begin{proof}
By Proposition \ref{r_c(3K_2,P_m)}, there exists a connected graph $F$ of size $\left\lceil \frac{3m+7}{2}\right\rceil$ such that $F\to (3K_2,P_m)$. Let $G=F+\frac{t-3}{2}C_{m+1}$. Since $C_{m+1}\to (2K_2,P_m)$, we have that $\frac{t-3}{2}C_{m+1}\to ((t-3)K_2,P_m)$, and so $G\to (tK_2,P_m)$. This proves that $\rhat(tK_2,P_m)\le \left\lceil\frac{3m+7}{2}\right\rceil+\frac{(t-3)(m+1)}{2}$.

Let $H$ be the graph obtained by inserting $\frac{t-3}{2}$ bridges between the components of $G$ so that $H$ is connected. We clearly have $H\to (tK_2,P_m)$, which implies that $\rc(tK_2,P_m)\le \left\lceil\frac{3m+7}{2}\right\rceil +\frac{(t-3)(m+2)}{2}$.
\end{proof}

For odd $t \ge 3$, the upper bound for $\rhat(tK_2,P_m)$ (respectively, for $\rc(tK_2,P_m)$) obtained in Theorem \ref{r(tK_2,P_m) sharp} is $2m-\left\lceil\frac{3m+7}{2}\right\rceil$ smaller (respectively, $2m-\left\lceil\frac{3m+7}{2}\right\rceil+1$ smaller) than the upper bound found in Theorem \ref{r(tK_2,P_m)}. For values $m \ge 9$, this is a marked improvement.

\section{$2K_2$ versus $C_n$}

In this last section, we briefly discuss the pair of graphs $(2K_2,C_n)$. The following theorem provides an upper bound for $\rc(2K_2,C_n)$ by constructing graphs similar to the ones given in Figure \ref{Graph34}.

\begin{theorem}\label{r_c(2K_2,C_n)}
For $n \ge 6$,
\[\rc(2K_2,C_n) \le \begin{cases}
\frac{3n+4}{2}, & n \text{ even,} \\ \frac{3n+7}{2}, & n \text{ odd.}
\end{cases}\]
\end{theorem}

\begin{figure}
\centering
\captionsetup{width=.27\textwidth}
\subfloat[The graph $F$ satisfies $F\to (2K_2,C_n)$ when $n \ge 6$ is even.]{%
\includegraphics[width=.27\textwidth]{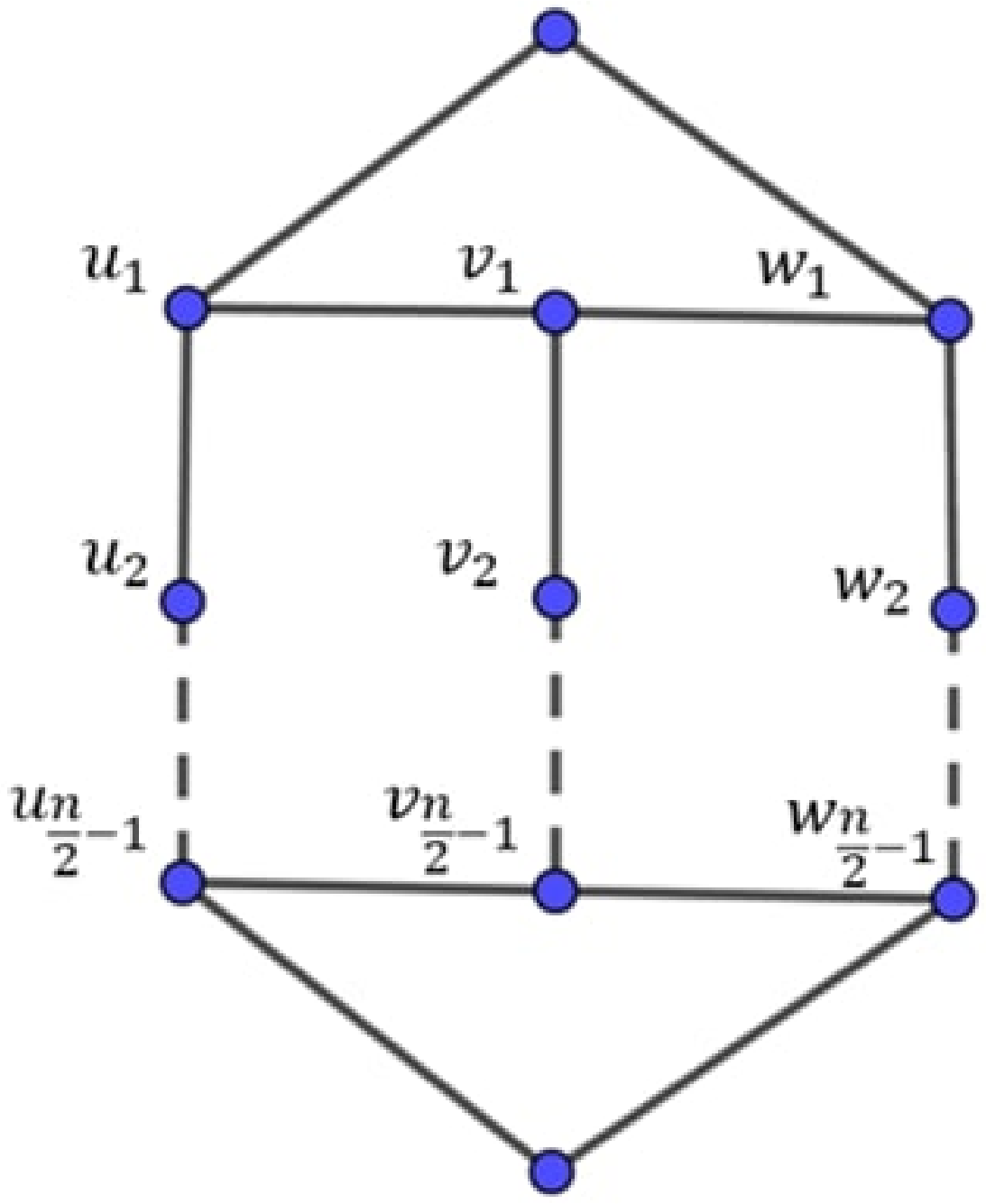}}\hspace{16pt}
\subfloat[The graph $G$ satisfies $G\to (2K_2,C_n)$ when $n \ge 7$ is odd.]{%
\includegraphics[width=.28\textwidth]{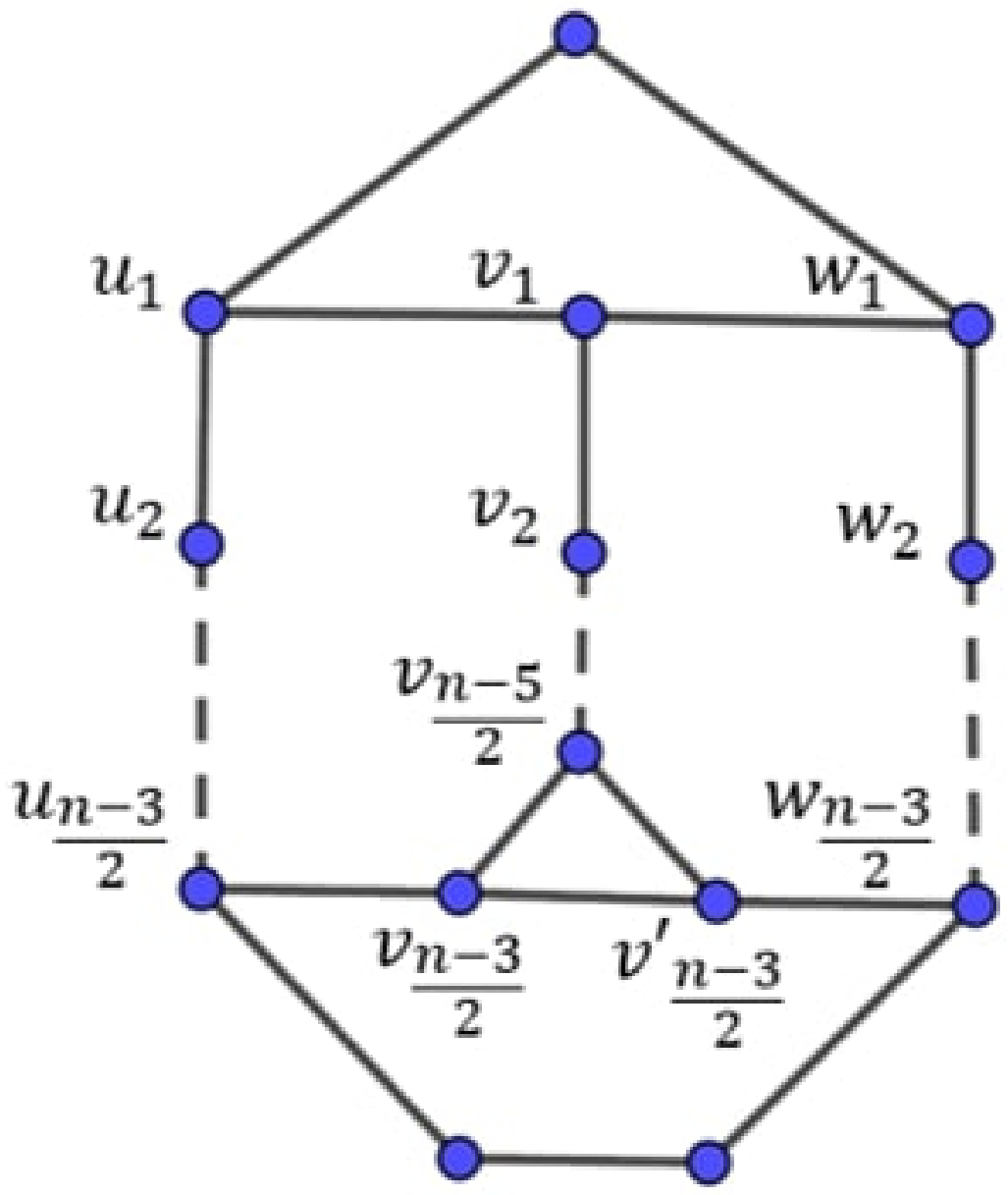}}
\captionsetup{width=.5\textwidth}
\caption{Arrowing graphs of $2K_2$ and $C_n$.} \label{Graph56}
\end{figure}

\begin{proof}
Suppose $n$ is even and let $F$ be the connected graph of size $\frac{3n+4}{2}$ shown in Figure \ref{Graph56}(A). We can verify that $C_n \subseteq F-v$ for every $v \in V(F)$. Since $F$ contains no triangle $C_3$, we have that $F\to (2K_2,C_n)$ by Lemma \ref{2K_2}. Therefore, $\rc(2K_2,C_n) \le \frac{3n+4}{2}$ when $n$ is even.

Suppose $n$ is odd and let $G$ be the connected graph of size $\frac{3n+7}{2}$ shown in Figure \ref{Graph56}(B). We can verify that $C_n \subseteq G-v$ for every $v \in V(G)$. We see that $G$ contains exactly one $C_3$, and it is easy to see that $C_n \subseteq G-C_3$. We thus have that $G\to (2K_2,C_n)$ by Lemma \ref{2K_2}, and that $\rc(2K_2,C_n) \le \frac{3n+7}{2}$ when $n$ is odd.
\end{proof}

Theorem \ref{r_c(2K_2,C_n)} refutes the claim made in \cite{Rah15} that $\rc(2K_2,C_n)=2n$, since it is able to provide a smaller upper bound than $2n$ for $\rc(2K_2,C_n)$, $n \ge 6$. Hence, we declare that the problem of finding the exact value of $\rc(2K_2,C_n)$ is, at present, still open.

\section{Concluding remarks}

We discussed the two types of size Ramsey numbers for the pair $(2K_2,nP_m)$. Furthermore, we managed to provide upper bounds for $\rc(3K_2,P_m)$ and $\rc(2K_2,C_n)$ when $m$ and $n$ are sufficiently large. In general, it seems to be difficult to obtain their exact values.

\begin{problem}
Find the exact values of $\rc(3K_2,P_m)$ and $\rc(2K_2,C_n)$ for $m,n\ge 3$.
\end{problem}

Working with the graph pair $(tK_2,nP_m)$ in its full generality also seems to be difficult. We invite future attempts at finding a good upper bound for $\rhat(tK_2,nP_m)$ and $\rc(tK_2,nP_m)$.

We have also explored the interplay between the original size Ramsey numbers and it connected variant. Theorem \ref{r_c(2K_2,nP_m)}, for example, which examines connected size Ramsey numbers, is used in the proofs of Propositions \ref{r(2K_2,2P_m)}--\ref{r(2K_2,4P_m)} on the exact values of size Ramsey numbers. Future research can be done to apply known results in the theory of connected size Ramsey numbers to problems regarding the more standard size Ramsey numbers.

\end{document}